\documentclass[reqno,11pt]{amsart}
\usepackage{latexsym,amssymb,amsthm,amsmath,graphics,amscd}

\newtheorem{theorem}{Theorem}[section]
\newtheorem{lemma}{Lemma}[section]
\newtheorem{proposition}{Proposition}[section]

\theoremstyle{remark}
\newtheorem{remark}{Remark}[section]
\theoremstyle{remark}

\theoremstyle{definition}
\newtheorem{definition}{Definition}[section]

\numberwithin{equation}{section}

\begin{document}

\title{Warped Product Pointwise Bi-slant Submanifolds of Kaehler Manifolds}
\author{Bang-Yen Chen}
\address{Department of Mathematics, 
	Michigan State University, 619 Red Cedar Road,  East Lansing, Michigan 48824--1027, U.S.A.}
\email{chenb@msu.edu}

\author{Siraj Uddin}
\address{ Department of Mathematics, Faculty of Science, King Abdulaziz University, 21589 Jeddah, Saudi Arabia}
\email{siraj.ch@gmail.com}



\begin{abstract} Warped product manifolds have been studied for a long period of time. In contrast, the study of warped product submanifolds from extrinsic point of view was initiated by the first author around the beginning of this century in \cite{C3,C4}.  Since then the study of  warped product submanifolds has been investigated by many geometers.

The notion of slant submanifolds of almost Hermitian manifolds was introduced in \cite{C1}. Bi-slant submanifolds in almost contact metric manifolds were
defined  in \cite{Cab} by J. L. Cabrerizo et al.  In \cite{U3}, we studied  bi-slant submanifolds and warped product bi-slant submanifolds in Kaehler manifolds. 
 In this  article, we investigate warped product pointwise bi-slant submanifolds of  Kaehler manifolds. Our main results extend several important results on warped product slant submanifolds obtained in [7, 21--23, 27].

\end{abstract}
\maketitle

\section{Introduction}
\label{sec1}
The notion of slant submanifolds was introduced by B.-Y. Chen in \cite{C1} and the first results on slant submanifolds were collected in his book \cite{C2}. Since then this subject have been studied extensively by many geometers during the last two and half decades. Many interesting results on slant submanifolds have been obtained in [5, 6].
As an extension of slant submanifolds, F. Etayo \cite{Etayo} defined the notion of pointwise slant submanifolds under the name of quasi-slant submanifolds. In \cite{Etayo} he proved that a complete totally geodesic quasi-slant submanifold of a Kaehler manifold is a slant submanifold. In \cite{C7}, the first author and O. J. Garay  studied pointwise slant submanifolds and proved many interesting new results on such submanifolds. In particular, they provided a method  to construct pointwise slant  submanifolds of some Euclidean spaces. 
 
Warped product manifolds have been studied for a long period of time (cf. e.g., \cite{Bi,book17}). In contrast, the study of warped product submanifolds from extrinsic point of view was only initiated around the beginning of this century in \cite{C3,C4}.  Since then the study of  warped product submanifolds has been investigated by many geometers  (see, e.g., 
[1, 9, 10, 13, 16, 18--27] among many others.   For the most up-to-date overview of this subject, see \cite{book17}).
 
J. L. Cabrerizo et al. studied in  \cite{Cab} bi-slant submanifolds of almost contact metric manifolds.   In \cite{U3} the authors investigated bi-slant submanifolds in Kaehler manifolds.  The authors also studied   in \cite{U3} warped product bi-slant submanifolds. In particular, they proved that a warped product bi-slant submanifold in a Kaehler manifold is either a Riemannian product of two slant submanifolds or a warped product  submanifold $M_\theta\times_f M_\perp$ such that $M_\theta$ is a $\theta$-slant submanifold and $M_\perp$ is a totally real submanifold. The later one was known as a hemi-slant warped product submanifold, which have been studied by B. Sahin in \cite{S2}.

In this  article, we study warped product pointwise bi-slant submanifolds of a Kaehler manifold as a natural extension of bi-slant submanifolds. Our main results extend several important results on warped product slant submanifolds obtained in [7, 21--23, 27].

\section{Preliminaries}
\label{sec:2}

Let $(\tilde M,J, g)$ be an almost Hermitian manifold with almost
complex structure $J$ and a Riemannian metric $g$ such that
\begin{equation}
\label{2.1} J^2=-I,
\;\; g(JX, JY) = g(X, Y),\;\; X, Y\in {\mathfrak X}(\tilde M), \end{equation}
where $I$ denotes the identity map and ${\mathfrak X}(\tilde M)$ is the space consisting of  vector fields tangent to $\tilde M$.
Let $\tilde\nabla$ be the Levi-Civita connection on $\tilde M$. 
If the almost complex structure $J$ satisfies
\begin{equation}
\label{2.3}
(\tilde\nabla_XJ)Y=0,\;\; X, Y\in {\mathfrak X}(\tilde M),
\end{equation}
 then $\tilde M$ is called a {\it{Kaehler manifold}}.

Let $M$ be a Riemannian manifold isometrically immersed in $\tilde M$. Then $M$ is called a {\it  complex} submanifold if $J(T_xM)\subseteq T_xM$ holds for  $x\in M$, where $T_xM$ is the tangent space of $M$ at $x$. And $M$ is called {\it{totally real}} if $J(T_xM)\subseteq T^\perp_xM$ holds for  $x\in M$, where $T^\perp_xM$ denotes the normal space of $M$ at $x$. 

Besides complex and totally real submanifolds, there are several important classes of submanifolds defined by the behavior of tangent bundle of the submanifold under the action of the almost complex structure  of the ambient space.
For example, a submanifold $M$ is called a {\it{CR-submanifold}}  if there is a complex distribution ${\mathcal{D}}:p\to {\mathcal{D}}_p\subset T_pM$ whose orthogonal complementary distribution ${\mathcal{D}}^\perp:p\to{\mathcal{D}}^\perp_p\subset T_pM$ is totally real, i.e., $J({\mathcal{D}}^\perp_p)\subset T_p^\perp M$ (cf. \cite{B}).

For a unit vector $ X$ tangent to  a submanifold $M$ of $\tilde M$,  the  angle $\theta(X)$ between $JX$ and $T_pM$ is called the Wirtinger angle of $X$.
The submanifold $M$ is called a {\it {slant  submanifold}} if the Wirtinger angle  $\theta(X)$ is constant on $M$, i.e., the Wirtinger angle is independent  of the choice of $X\in T_pM$ and of $p\in M$   (cf.  \cite{C1,C2,book}). In this case, the constant angle $\theta$ is called the {\it{slant angle}} of the slant submanifold. A slant submanifold is called {\it proper} if its slant angle $\theta$ satisfying $\theta\ne 0,\frac{\pi}{2}$. Similar definitions apply to distributions.

A submanifold $M$ is called {\it{semi-slant}} if there is a pair of orthogonal distributions ${\mathcal{D}}$ and ${\mathcal{D}}^\theta$ such that ${\mathcal{D}}$ is complex and ${\mathcal{D}}^\theta$ is proper slant (cf.  \cite{P}).

A submanifold $M$ of  $\tilde M$ is called {\it{bi-slant}} if there exist two orthogonal distribution ${\mathcal{D}}_1$ and ${\mathcal{D}}_2$  on $M$ such that  $TM={\mathcal{D}}_1\oplus{\mathcal{D}}_2$ and ${\mathcal{D}}_1$ and $ {\mathcal{D}}_2$ are proper slant distributions satisfying $J{\mathcal{D}}_i\perp {\mathcal{D}}_j$ for $1\leq \,i\ne j\leq 2$ (cf.  \cite{U3}).

For a submanifold $M$ of a Riemannian manifold $\tilde M$, the formulas of Gauss and Weingarten are given respectively by
\begin{equation}
\label{2.4}
\tilde \nabla_X Y=\nabla_X Y+h(X,Y),
\end{equation}
\begin{equation}
\label{2.5}\tilde\nabla_XN=-A_NX+\nabla^\perp_XN
\end{equation}
for $X, Y\in TM$ and for normal vector field $N$ of $M$, where  $\nabla$ is the induced
Levi-Civita connection on $M$, $h$  the second fundamental form, $\nabla^\perp$  the
normal connection, and $A$  the shape operator. 
The shape operator and the second fundamental form of $M$ are related by 
\begin{equation}
\label{2.6}
g(A_NX,Y)= g(h(X,Y),N),
\end{equation}
where $g$ denotes the induced  metric on $M$ as well as the metric on $\tilde M$. 

For a tangent vector field $X$ and a normal vector field $N$ of $M$, we put
\begin{equation}
\label{2.7}JX=TX+FX,\;\;  JN=BN+CN,
\end{equation}
where $TX$ and $FX$ (respectively, $BN$ and $CN$) are the tangential and the normal components of $JX$ (respectively, of $JN$).

\begin{definition} A submanifold $M$ of an almost Hermitian manifold $\tilde M$ is called {\it{pointwise slant}} if, at each point $p\in M$, the Wirtinger angle $\theta(X)$ is independent of the choice of nonzero vector $X\in T^{*}_pM$, where $T^{*}_pM$ is the tangent space of nonzero vectors. In this case, $\theta$  is called {\it{slant function}} of $M$ (cf. \cite{C7}).
\end{definition}

The following is a simple characterization of pointwise slant submanifolds.

\begin{lemma}\label{L:1} {\rm \cite{C7}} Let $M$ be a submanifold of an almost Hermitian manifold $\tilde M$. Then $M$ is a pointwise slant submanifold if and only if 
\begin{equation}
\label{2.9}
T^2=-(\cos^2\theta) I
\end{equation}
for some real valued function $\theta$ defined on the tangent bundle $TM$ of $M$.
\end{lemma} 

Similarly, we can prove the following in a similar way as \cite{C7}.

\begin{proposition}\label{P:1} Let ${\mathcal{D}}$ be a distribution on a submanifold $M$. Then ${\mathcal{D}}$ is pointwise slant if and only if there is a constant $\lambda\in[-1, 0]$ such that $(PT)^2X=-\lambda X$, for any $X\in{\mathcal{D}}_p$ at  $p\in M$, where $P$ is the projection onto ${\mathcal{D}}$. Furthermore, in this case $\lambda=\cos^2\theta_{{\mathcal{D}}}$.
\end{proposition}

As easy consequences of the relation (\ref{2.9}), we find
\begin{equation}
\label{2.10}
g(TX,TY)=(\cos^2\theta) g(X,Y),
\;\; 
g(FX,FY)=(\sin^2\theta) g(X,Y).
\end{equation}
 Also, for a pointwise slant submanifold, (\ref{2.7}) and (\ref{2.9}) yield
\begin{equation}
\label{2.12}
BFX=-(\sin^2\theta) X\;\; {\mbox{and}}\;\; CFX=-FTX.
\end{equation}

\section{Pointwise Bi-slant Submanifolds}\label{sec3}

Now, we define pointwise bi-slant submanifolds. 
\begin{definition}
{\rm A submanifold $M$ of an almost Hermitian manifold $(\tilde M,J,g)$ is called  {\it{pointwise bi-slant}} if there exists a pair of orthogonal distributions ${\mathcal{D}}_1$ and ${\mathcal{D}}_2$  of $M$, at the point $p\in M$ such that
\vskip.05in

\noindent (a) $TM= {\mathcal{D}}_1\oplus{\mathcal{D}}_2$;

\noindent (b) $J{\mathcal{D}}_1 \perp {\mathcal{D}}_2$ and $J{\mathcal{D}}_2 \perp {\mathcal{D}}_1$;

\noindent (c) The distributions  ${\mathcal{D}}_1$, ${\mathcal{D}}_2$ are pointwise slant with slant function $\theta_1$, $\theta_2$, respectively.
\vskip.05in

\noindent The pair $\{\theta_1,\theta_2\}$ of slant functions is called the {\it bi-slant function}.
 A pointwise bi-slant submanifold $M$ is called {\it proper} if its bi-slant function satisfies $\theta_1,\theta_2\ne 0,\frac{\pi}{2}$ and both $\theta_1,\theta_2$ are not constant on $M$.}
\end{definition} 

Notice that \eqref{2.7} and condition (b) in Definition 3.1 imply that 
\begin{equation}\label{3.1} T({\mathcal{D}}_i)\subset {\mathcal{D}}_i,\;\; i=1,2.\end{equation}
Given a pointwise bi-slant submanifold,  for any $X\in TM$ we put
\begin{equation} \label{3.2}X=P_1X+P_2X
\end{equation}
where $P_i$ is the projection from $TM$ onto ${\mathcal{D}}_i$. Clearly, $P_iX$ is the components of $X$ in ${\mathcal{D}}_i$, $i=1,2$. In particular, if $X\in {\mathcal{D}}_i$, we have $X=P_iX$. 

If we put $T_i=P_i\circ T$, then we find from (\ref{3.2}) that
\begin{equation} \label{3.3}JX=T_1X+T_2X+FX
\end{equation}
for $X\in TM$. From Proposition \ref{P:1} we get
\begin{equation} \label{3.4}T_i^2X=-\big(\cos^2\theta_i\big)X,\;\; X\in TM,\;\; i=1,2.
\end{equation}

From now on, we assume the ambient manifold $\tilde M$ is Kaehlerian and  $M$ is pointwise bi-slant in $\tilde M$. 

We need the following lemma for later use.

\begin{lemma}\label{L:2} Let $M$ be a pointwise bi-slant submanifold of a Kaehler manifold $\tilde M$ with pointwise slant distributions ${\mathcal{D}}_1$ and ${\mathcal{D}}_2$ with distinct  slant functions $\theta_1$ and $\theta_2$, respectively. Then
\begin{enumerate}
\item[{\rm (i)}] For any $X, Y\in {\mathcal{D}}_1$ and $Z\in {\mathcal{D}}_2$, we have
\begin{align} \label{3.5}
\notag\left(\sin^2\theta_1-\sin^2\theta_2\right)\,g(\nabla_XY, Z)&=g(A_{FT_2Z}Y-A_{FZ}T_1Y, X)\\
&+g(A_{FT_1Y}Z-A_{FY}T_2Z, X).
\end{align}
\item[{\rm (ii)}] For $Z, W\in {\mathcal{D}}_2 $ and $X\in {\mathcal{D}}_1$, we have
\begin{align} \label{3.6}
\notag\left(\sin^2\theta_2-\sin^2\theta_1\right)\,g(\nabla_ZW, X)&=g(A_{FT_2W}X-A_{FW}T_1X, Z)\\
&+g(A_{FT_1X}W-A_{FX}T_2W, Z).
\end{align}
\end{enumerate}
\end{lemma}
\begin{proof}  For  $X, Y\in {\mathcal{D}}_1$ and $Z\in {\mathcal{D}}_2$, we have
\begin{align}
\notag g(\nabla_XY, Z)=g(\tilde\nabla_XY, Z)=g(\tilde\nabla_XJY, JZ).
\end{align}
From (\ref{2.7}) we derive
\begin{align}
g(\nabla_XY, Z)=&\,g(\tilde\nabla_XT_1Y, JZ)+g(\tilde\nabla_XFY, T_2Z)+g(\tilde\nabla_XFY, FZ)\notag\\
=&-g(\tilde\nabla_XT^2_1Y, Z)-g(\tilde\nabla_XFT_1Y, Z)-g(A_{FY}X, T_2Z)\notag\\
&-g(\tilde\nabla_XFZ, FY).\notag
\end{align}
Again, using (\ref{2.7}) and (\ref{3.4}), we arrive at
\begin{align}
g(\nabla_XY, Z)=&\cos^2\theta_1\,g(\tilde\nabla_XY, Z)-\sin2\theta_1X(\theta_1)\,g(Y, Z)+g(A_{FT_1Y}X, Z)\notag\\
&-g(A_{FY}X, T_2Z)-g(\tilde\nabla_XFZ, JY)+g(\tilde\nabla_XFZ, T_1Y).\notag
\end{align}
By the orthogonality of two distributions and the symmetry of the shape operator, the above equation takes the from
\begin{align}
\sin^2\theta_1\,g(\nabla_XY, Z)&=g(A_{FT_1Y}Z-A_{FY}T_2Z, X)+g(\tilde\nabla_XBFZ, Y)\notag\\
&+g(\tilde\nabla_XCFZ, Y)-g(A_{FZ}X, T_1Y).\notag
\end{align}
Then we find from (\ref{2.12}) that
\begin{align}
\sin^2\theta_1\,g(\nabla_XY, Z)&=g(A_{FT_1Y}Z-A_{FY}T_2Z, X)-\sin^2\theta_2\,g(\tilde\nabla_XZ, Y)\notag\\
&\hskip-.3in -\sin2\theta_2X(\theta_2)\,g(Y, Z)-g(\tilde\nabla_XFT_2Z, Y)-g(A_{FZ}T_1Y, X).\notag
\end{align}
Using (\ref{2.5}) and the orthogonality of vector fields, we get
\begin{align}
\sin^2\theta_1\,g(\nabla_XY, Z)&=g(A_{FT_1Y}Z-A_{FY}T_2Z, X)+\sin^2\theta_2\,g(\tilde\nabla_XY, Z)\notag\\
&+g(A_{FT_2Z}X, Y)-g(A_{FZ}T_1Y, X).\notag
\end{align}
Now, part (i) of the lemma follows from above relation by using the symmetry of the shape operator. In a similar way, we can prove (ii).
\end{proof}

\section{Warped Product Pointwise Bi-slant Submanifolds}\label{sec4}

Let $B$ and $F$ be two Riemannian manifolds with metrics $g_B$ and $g_F$, respectively, and $f$ a smooth function on $B$. Consider the product manifold $B\times F$ with projections $\pi_1:B\times F\to B$ and $\pi_2:B\times F\to F$. The warped product $M=B\times_fF$ is the manifold equipped with the Riemannian metric given by
\begin{equation*}
g(X, Y)=g_B({\pi_1}_\star X, {\pi_1}_\star Y)+(f\circ\pi_1)^2g_F({\pi_2}_\star X, {\pi_2}_\star Y)
\end{equation*}
for  $X, Y\in {\mathfrak  X}(M)$, where $\star$ denotes the tangential maps. 

A warped product $M_1\times _fM_2$ is called {\it{trivial}} (or simply called a {\it{Riemannian product}}) if the warping function $f$ is constant. Let $X$ be a vector field tangent to $M_1$ and $Z$  a  vector field tangent to $M_2$, then Lemma 7.3 of \cite{Bi} gives
\begin{equation}
\label{4.1}
\nabla_XZ=\nabla_ZX=X(\ln f)Z
\end{equation}
where $\nabla$ is the Levi-Civita connection on $M$. 

 For a warped product $M=M_1\times_fM_2$, the base manifold $M_1$ is totally geodesic in $M$ and the fiber $M_2$ is totally umbilical in $M$ (see \cite{Bi,C3}). 

In this section, we study warped product pointwise bi-slant submanifolds in a Kaehler manifold $\tilde M$.

First, we give the following lemmas for later use.

\begin{lemma}\label{L:3} Let $M_1\times_fM_2$ be a warped product pointwise bi-slant submanifold of a Kaehler manifold $\tilde M$ such that $M_1$ and $M_2$ are pointwise slant submanifolds with slant functions $\theta_1$ and $\theta_2$, respectively of $\tilde M$. Then
\begin{align}
\label{4.2}
g(h(X, W), FT_2Z)-g(h(X, T_2Z), FW)=(\sin2\theta_2)X(\theta_2)\,g(Z, W)
\end{align}
for any $X\in TM_1$ and $Z, W\in TM_2$.
\end{lemma}
\begin{proof} For any $X\in TM_1$ and $Z, W\in TM_2$, we have
\begin{align}
\label{4.3}
g(\tilde\nabla_XZ, W)=g(\nabla_XZ, W)=X(\ln f)\,g(Z, W).
\end{align}
On the other hand, we also have
\begin{align*}
g(\tilde\nabla_XZ, W)=g(J\tilde\nabla_XZ, JW)=g(\tilde\nabla_XJZ, JW)
\end{align*}
for any $X\in TM_1$ and $Z, W\in TM_2$. Using (\ref{2.7}), we obtain
\begin{align*}
g(\tilde\nabla_XZ, W)=g(\tilde\nabla_XT_2Z, T_2W)+g(\tilde\nabla_XT_2Z, FW)+g(\tilde\nabla_XFZ, JW).
\end{align*}
Then from (\ref{2.1}), (\ref{2.3}),  (\ref{2.4}) and (\ref{4.1}), we derive
\begin{align*}
g(\tilde\nabla_XZ, W)&=\cos^2\theta_2\,X(\ln f)\,g(Z, W)+g(h(X, T_2Z), FW)-g(\tilde\nabla_XJFZ, W)\\
&=\cos^2\theta_2\,X(\ln f)\,g(Z, W)+g(h(X, T_2Z), FW)-g(\tilde\nabla_XBFZ, W)\\
&\hskip.2in -g(\tilde\nabla_XCFZ, W).
\end{align*}
Using (\ref{2.12}), we find
\begin{align}
\label{4.4}
&g(\tilde\nabla_XZ, W)=\cos^2\theta_2\,X(\ln f)\,g(Z, W)+g(h(X, T_2Z), FW)\notag\\
&\hskip.2in+\sin^2\theta_2\,g(\tilde\nabla_XZ, W)+\sin2\theta_2\,X(\theta_2)\,g(Z, W)+g(\tilde\nabla_XFT_2Z, W).
\end{align}
Thus the lemma follows from (\ref{4.3}) and (\ref{4.4}) by using  (\ref{2.5}) and (\ref{4.1}).\end{proof}

\begin{lemma}\label{L:4} Let $M_1\times_fM_2$ be a warped product pointwise bi-slant submanifold of a Kaehler manifold $\tilde M$ such that $M_1$ and $M_2$ are pointwise slant submanifolds with slant functions $\theta_1$ and $\theta_2$, respectively of $\tilde M$. Then
\begin{align}
\label{4.5}
g(h(X, Z), FW)-g(h(X, W), FZ)=2(\tan\theta_2)X(\theta_2)\,g(T_2Z, W)
\end{align}
for any $X\in TM_1$ and $Z, W\in TM_2$.
\end{lemma}
\begin{proof} By interchanging $Z$ by $T_2 Z$ in (\ref{4.2}) for any $Z\in TM_2$ and by using (\ref{3.4}), we obtain the required result.
\end{proof}

\begin{lemma}\label{L:5} Let $M_1\times_fM_2$ be a warped product pointwise bi-slant submanifold of a Kaehler manifold $\tilde M$ such that $M_1$ and $M_2$ are pointwise slant submanifolds with slant functions $\theta_1$ and $\theta_2$, respectively of $\tilde M$. Then
\begin{align}
\label{4.6}
g(h(X, W), FT_2Z)-g(h(X, T_2Z), FW)=(\cos^2\theta_2)X(\ln f)\,g(Z, W)
\end{align}
for any $X\in TM_1 $ and $Z, W\in TM_2$.
\end{lemma}
\begin{proof} For any $X\in TM_1$ and $Z, W\in TM_2$, we have
\begin{align*}
g(h(X, Z), FW)=g(\tilde\nabla_ZX, FW)=g(\tilde\nabla_ZX, JW)-g(\tilde\nabla_ZX, T_2W).
\end{align*}
Using (\ref{2.1}), (\ref{2.3}), (\ref{2.7}) and (\ref{4.1}), we get
\begin{align*}
g(h(X, Z), FW)=-g(\tilde\nabla_ZT_1X, W)-g(\tilde\nabla_ZFX, W)-X(\ln f)\,g(Z, T_2W).
\end{align*}
Again from (\ref{2.1}), (\ref{4.1}) and (\ref{2.5})-(\ref{2.6}), we arrive at
\begin{align}
\label{4.7}
g(h(X, Z), FW)&=-T_1X(\ln f)\,g(Z, W)+g(h(Z, W), FX)\notag\\
&+X(\ln f)\,g(T_2Z, W).
\end{align}
Then from polarization, we derive
\begin{align}
\label{4.8}
g(h(X, W), FZ)&=-T_1X(\ln f)\,g(Z, W)+g(h(Z, W), FX)\notag\\
&-X(\ln f)\,g(T_2Z, W).
\end{align}
Subtracting (\ref{4.8}) from (\ref{4.7}), we obtain
\begin{align}
\label{4.9}
g(h(X, Z), FW)-g(h(X, W), FZ)=2X(\ln f)\,g(T_2Z, W).
\end{align}
Interchanging $Z$ by $T_2Z$ in (\ref{4.9}) and using (\ref{3.4}), we get (\ref{4.6}), which proves the lemma completely.\end{proof}

A warped product submanifold $M_1\times_fM_2$ of a Kaehler manifold $\tilde M$ is called {\it{mixed totally geodesic}} if $h(X, Z)=0$ for any $X\in TM_1$ and $Z\in TM_2$.

Now, by applying Lemma \ref{L:5}, we obtain following theorem.

\begin{theorem}\label{T:1} Let $M=M_1\times_fM_2$ be a warped product pointwise bi-slant submanifold of a Kaehler manifold $\tilde M$ such that $M_1$ and $M_2$ are pointwise slant submanifolds with slant functions $\theta_1$ and $\theta_2$, respectively of $\tilde M$. Then, if $M$ is a mixed totally geodesic warped product submanifold, then one of the two following cases occurs:
\begin{enumerate}
\item [{\rm (i)}] either $M$ is a Riemannian product submanifold of $M_1$ and $M_2$,
\item [{\rm (ii)}] or $\theta_2=\frac{\pi}{2}$, i.e., $M$ is a warped product submanifold of the from $M_1\times_fM_\perp$
\end{enumerate}
where $M_\perp$ is a totally real submanifold of $\tilde M$.
\end{theorem}
\begin{proof} From Lemma \ref{L:5} and mixed totally  geodesic condition, we have
\begin{align*}
(\cos^2\theta_2)X(\ln f)\,g(Z, W)=0,
\end{align*}
which shows that either $f$ is constant on $M$ or $\cos^2\theta=0$. Hence either $M$ is a Riemannian product or $\theta_2=\frac{\pi}{2}$. This completes the proof of the theorem.
\end{proof}

\begin{remark} In Theorem \ref{T:1}, if $M$ is mixed totally geodesic and $f$ is not constant on $M$, then $M$ is a warped product pointwise hemi-slant submanifold of the form $M_\theta\times_fM_\perp$, where $M_\theta$ is a pointwise slant submanifold and $M_\perp$ is a totally real submanifold of $\tilde M$. These kinds of warped product are special case of warped product hemi-slant submanifolds which have been discussed in \cite{S2}, therefore we are not interested to study mixed geodesic case.
\end{remark}

Now,  we have the following useful result.

\begin{theorem}\label{T:2} Let $M=M_1\times_fM_2$ be a warped product pointwise bi-slant submanifold of a Kaehler manifold $\tilde M$ such that $M_1$ and $M_2$ are proper pointwise slant submanifolds with slant functions $\theta_1$ and $\theta_2$, respectively of $\tilde M$. Then
\begin{align}
\label{4.10}
X(\ln f)=(\tan\theta_2)X(\theta_2)
\end{align}
for any $X\in TM_1$.
\end{theorem}
\begin{proof} From Lemma \ref{L:3} and Lemma \ref{L:5}, we have
\begin{align*}
\cos^2\theta_2\,X(\ln f)\,g(Z, W)=\sin2\theta_2\,X(\theta_2)\,g(Z, W),
\end{align*}
for any $X\in TM_1 $ and any $Z, W\in TM_2$. Using trigonometric identities, we find
$\{X(\ln f)-\tan\theta_2\,X(\theta_2)\}\,g(Z, W)=0$,
which implies $X(\ln f)=\tan\theta_2\,X(\theta_2)$. This proves the theorem.
\end{proof}

We have the following immediate consequences of above theorem:

\vskip.1in

\noindent 1. If $\theta_1=0$ and $\theta_2=\theta\ne \frac{\pi}{2}$ is a constant, then warped product is of the form $M_T\times_fM_\theta$, which is a semi-slant warped product submanifold. In this case it follows from Theorem \ref{T:2} that $X(\ln f)=0$. Thus $f$ is constant. Consequently, Theorem 3.2 of \cite{S1} is an special case of Theorem \ref{T:2}.
\vskip.05in

\noindent 2. In a pointwise bi-slant submanifold $M_1\times_fM_2$, if $\theta_2=0$, then the warped product is of the form $M_\theta\times_fM_T$, where $M_T$ is a complex submanifold and $M_\theta$ is a pointwise slant submanifold with slant function $\theta$. In this case, it also follows from Theorem \ref{T:2} that $f$ is constant. Thus Theorem \ref{T:2} is also a generalization of Theorem 4.1 in \cite{S3}.
\vskip.05in

\noindent 3. If $\theta_1=\frac{\pi}{2}$ and $\theta_2$ is a constant $\theta$, then the warped product pointwise bi-slant is of the from $M_\perp\times_fM_\theta$, which is a hemi-slant warped product. Such submanifolds were discussed in \cite{S2}. In this case  Theorem \ref{T:2} also implies that $f$ is constant. Thus Theorem 4.2 of \cite{S2} is a special case of Theorem \ref{T:2} as well.
\vskip.05in

\noindent4. Again, if $\theta_1=\frac{\pi}{2}$ and $\theta_2=0$, then the warped product pointwise bi-slant submanifold becomes a warped product CR-submanifold $M_\perp\times_fM_T$, and in this case we know from Theorem \ref{T:2} that $f$ is constant. Thus Theorem \ref{T:2} is also a generalization of Theorem 3.1 in \cite{C3}.
\vskip.05in

\noindent 5. If $\theta_1$ and $\theta_2$ are constant, then the warped product $M=M_1\times_fM_2$ is a warped product bi-slant submanifold and in this case Theorem \ref{T:2} also implies that $f$ is constant. Thus Theorem 5.1 of \cite{U3} is also a special case of Theorem \ref{T:2}.
\begin{remark} It is clear from Theorem \ref{T:2} that there exist no  warped product pointwise bi-slant submanifolds of the forms $M_1\times_fM_T$ or $M_1\times_fM_\theta$, where $M_1$ is a pointwise slant submanifold and $M_T$ and $M_\theta$ are complex and proper slant submanifolds of $\tilde M$, respectively.
\end{remark}

We also need the next  lemma.

\begin{lemma}\label{L:6} Let $M=M_1\times_fM_2$ be a warped product pointwise bi-slant submanifold of a Kaehler manifold $\tilde M$ such that $M_1$ and $M_2$ are proper pointwise slant submanifolds with distinct slant functions $\theta_1$ and $\theta_2$, respectively of $\tilde M$. Then we have
\vskip.1in

\noindent {\rm (i)} \hskip.9in $g(h(X, Y), FZ)=g(h(X, Z), FY),$ 
\begin{align}\notag& {\rm (ii)} \hskip.5in 
g(A_{FT_1X}W-A_{FX}T_2W, Z)+g(A_{FT_2W}X-A_{FW}T_1X, Z)  \hskip.9in\\ \notag
&\hskip1.3in =\big(\sin^2\theta_1-\sin^2\theta_2\big)X(\ln f)g(Z, W),
\end{align}
for any $X, Y\in TM_1 $ and $Z, W\in TM_2$.
\end{lemma}
\begin{proof} Part (i) is trivial and it can be obtained by using Gauss-Weingarten formulas, relation (\ref{4.1}) and orthogonality of vector fields. For (ii), we have
\begin{align}
\label{4.11}
g(\tilde\nabla_ZX, W)=g(\nabla_ZX, W)=X(\ln f)\,g(Z, W)
\end{align}
for any $X, Y\in TM_1$ and $Z, W\in TM_2$. On the other hand, we have
\begin{align*}
g(\tilde\nabla_ZX, W)=g(J\tilde\nabla_ZX, JW)=g(\tilde\nabla_ZJX, JW).
\end{align*}
Then from (\ref{2.7}), we get
\begin{align*}
g(\tilde\nabla_ZX, W)=g(\tilde\nabla_ZT_1X, JW)+g(\tilde\nabla_ZFX, T_2W)+g(\tilde\nabla_ZFX, FW).
\end{align*}
Using (\ref{2.1}), (\ref{2.3}), (\ref{2.5}) and covariant derivative property of the metric connection, we get
\begin{align*}
g(\tilde\nabla_ZX, W)=-g(\tilde\nabla_ZJT_1X, W)-g(A_{FX}Z, T_2W)-g(\tilde\nabla_ZFW, FX).
\end{align*}
From (\ref{2.7}) and the symmetry of the shape operator, we derive
\begin{align*}
g(\tilde\nabla_ZX, W)&=-g(\tilde\nabla_ZT^2_1X, W)-g(\tilde\nabla_ZFT_1X, W)-g(A_{FX}T_2W, Z)\\
&+g(J\tilde\nabla_ZFW, X)+g(\tilde\nabla_ZFW, T_1X)\\
&=\cos^2\theta_1\,g(\tilde\nabla_ZX, W)-\sin2\theta_1\,Z(\theta_1)\,g(X, W)+g(A_{FT_1X}Z, W)\\
&-g(A_{FX}T_2W, Z)+g(\tilde\nabla_ZJFW, X)-g(A_{FW}Z, T_1X).
\end{align*}
Using (\ref{2.4}), (\ref{2.7}), (\ref{4.1}), (\ref{4.11}) and the orthogonality of vector fields and symmetry of the shape operator, we obtain
\begin{align*}
\sin^2\theta_1\,&X(\ln f)\,g(Z, W)=g(A_{FT_1X}W-A_{FX}T_2W, Z)\\
&+g(\tilde\nabla_ZBFW, X)+g(\tilde\nabla_ZCFW, X)-g(A_{FW}T_1X, Z).
\end{align*}
Using (\ref{2.12}), we arrive at
\begin{align*}
\sin^2\theta_1\,&X(\ln f)\,g(Z, W)=g(A_{FT_1X}W-A_{FX}T_2W, Z)-\sin^2\theta_2\,g(\tilde\nabla_ZW, X)\\
&-\sin2\theta_2\,Z(\theta_2)\,g(X, W)-g(\tilde\nabla_ZFT_2W, X)-g(A_{FW}T_1X, Z).
\end{align*}
From the orthogonality of vector fields and the relations (\ref{2.4}), (\ref{2.5}) and (\ref{4.1}), we find that
\begin{align*}
\sin^2\theta_1\,&X(\ln f)\,g(Z, W)=g(A_{FT_1X}W-A_{FX}T_2W, Z)\\&+\sin^2\theta_2\,X(\ln f)\,g(Z, W)
+g(A_{FT_2W}Z, X)-g(A_{FW}T_1X, Z).
\end{align*}
Again, using the symmetry of the shape operator, we get (ii) from the above relation. Hence the lemma is proved completely.\end{proof}

A foliation $L$ on a Riemannian manifold $M$ is called  {\it totally umbilical}, if every leaf  of $L$ is totally umbilical in $M$. If, in addition, the mean curvature vector of every leaf is parallel in the normal bundle,  then $L$ is called  a {\it  spheric foliation}. 
  If every leaf of $L$ is totally geodesic, then $L$ is called a {\it totally geodesic foliation} (cf. [11, 14, 17]).

We need the following well known result of S. Hiepko \cite{Hi}.
\vskip.1in

\noindent{\bf{Hiepko's Theorem.}} {\it{Let ${\mathcal{D}}_1$ and ${\mathcal{D}}_2$ be two orthogonal distribution on a Riemannian manifold $M$. Suppose that ${\mathcal{D}}_1$ and ${\mathcal{D}}_2$ both are involutive such that ${\mathcal{D}}_1$ is a totally geodesic foliation and ${\mathcal{D}}_2$ is a spherical foliation. Then $M$ is locally isometric to a non-trivial warped product $M_1\times_fM_2$, where $M_1$ and $M_2$ are integral manifolds of ${\mathcal{D}}_1$ and ${\mathcal{D}}_2$ , respectively.}}
\vskip.1in

The following result provides a characterization of warped product pointwise bi-slant submanifolds of a Kaehler manifold.

\begin{theorem}\label{T:3} Let $M$ be a proper pointwise bi-slant submanifold of a Kaehler manifold $\tilde M$ with pointwise slant distributions ${\mathcal{D}}_1$ and ${\mathcal{D}}_2$. Then $M$ is locally a warped product submanifold of the form $M_1\times_fM_2$, where $M_1$ and $M_2$ are pointwise slant submanifolds with distinct slant functions $\theta_1$ and $\theta_2$, respectively of $\tilde M$ if and only if the shape operator of $M$ satisfies
\begin{align}
\label{4.12} A_{FT_1X}Z -A_{FX}T_2Z+A_{FT_2Z}X-A_{FZ}T_1X=\left(\sin^2\theta_1-\sin^2\theta_2\right)X(\mu)\,Z
\end{align}
for $X\in {\mathcal{D}}_1$, $Z\in {\mathcal{D}}_2,$ and for a function $\mu$ on $M$ satisfying $W\mu=0$ for any $W\in {\mathcal{D}}_2$.
\end{theorem}
\begin{proof} Let $M=M_1\times_fM_2$ be a pointwise bi-slant submanifold of a Kaehler manifold $\tilde M$. Then from Lemma \ref{L:6}(i), we have
\begin{align}\label{4.13} g(A_{FY}Z-A_{FZ}Y, X)=0
\end{align}
for any $X, Y\in TM_1$ and $Z\in TM_2$.
Interchanging $Y$ by $T_1Y$ in (\ref{4.13}), we get
\begin{align}
\label{4.14}
g(A_{FT_1Y}Z-A_{FZ}T_1Y, X)=0.
\end{align}
Again interchanging $Z$ by $T_2Z$ in (\ref{4.13}), we obtain
\begin{align}
\label{4.15}
g(A_{FY}T_2Z-A_{FT_2Z}Y, X)=0.
\end{align}
Subtracting (\ref{4.15}) from (\ref{4.14}), we derive
\begin{align}
\label{4.16}
g(A_{FT_1Y}Z-A_{FZ}T_1Y+A_{FT_2Z}Y-A_{FY}T_2Z, X)=0.
\end{align}
Then (\ref{4.12}) follows from Lemma \ref{L:6}(ii) by using the above fact.

Conversely, if $M$ is a pointwise bi-slant submanifold with pointwise slant distributions ${\mathcal{D}}_2$ and ${\mathcal{D}}_2$ such that (\ref{4.12}) holds, then from Lemma 
\ref{L:2}(i), we have
\begin{align*} (\sin^2\theta_1-\sin^2\theta_2)
g(\nabla_XY, Z)&=g(A_{FT_1Y}Z-A_{FY}T_2Z+A_{FT_2Z}Y-A_{FZ}T_1Y, X)
\end{align*}
for any $X, Y\in {\mathcal{D}}_1$ and $Z\in {\mathcal{D}}_2$. Using the given condition (\ref{4.12}), we get
\begin{align*}
g(\nabla_XY, Z)&=X(\mu)\,g(X, Z)=0,
\end{align*}
which shows that the leaves of the distributions are totally geodesic in $M$. On the other hand, from Lemma \ref{L:2}(ii) we have
\begin{align*}
\notag (\sin^2\theta_2-\sin^2\theta_1)g(\nabla_ZW, X)&=g(A_{FT_2W}X-A_{FW}T_1X\\
&+A_{FT_1X}W-A_{FX}T_2W, Z).
\end{align*}
From the hypothesis of the theorem i.e., (\ref{4.12}), we arrive at
\begin{align}
\label{4.17}
g(\nabla_ZW, X)=-X(\mu)\,g(Z, W).
\end{align}
By polarization, we obtain
\begin{align}
\label{4.18}
g(\nabla_WZ, X)=-X(\mu)\,g(Z, W).
\end{align}
Subtracting (\ref{4.18}) from (\ref{4.17}) and using the definition of Lie bracket, we derive
$ g([Z, W], X)=0$,
which shows that the distribution ${\mathcal{D}}_2$ is integrable. If we consider a leaf $M_2$ of ${\mathcal{D}}_2$ and the second fundamental form $h_2$ of $M_2$ in $M$, then from (\ref{4.17}), we have
\begin{align*}
\notag g(h_2(Z, W), X)=g(\nabla_ZW, X)=-X(\mu)\,g(Z, W).
\end{align*}
Using the definition of the gradient we get
$h_2(Z, W)=-\vec\nabla\mu\,g(Z, W),$
where $\vec\nabla\mu$ is the gradient of  $\mu$. The above relations shows that the leaf $M_2$ is totally umbilical in $M$ with mean curvature vector $H_2=-\vec\nabla\mu$. Since $W(\mu)=0$ for any $W\in {\mathcal{D}}_2$, it is easy to see that the mean curvature is parallel. Hence the spherical condition is satisfied. Then, by Hiepko's Theorem $M$ is locally a warped product submanifold. Hence the proof is complete.
\end{proof}

We have the following consequences of the above theorem:
\vskip.05in

\noindent 1. In Theorem \ref{T:3}, if $\theta_1=0$ and $\theta_2=\frac{\pi}{2}$, then all terms in the left hand side of (\ref{4.12}) vanish identically except the last term, thus the relation (\ref{4.12}) is valid for $CR$-warped product and it will be $$A_{JZ}JX=-X(\mu)\,Z,\,\,\,\forall\,X\in {\mathcal{D}} ,\,\,Z\in {\mathcal{D}}^\perp,$$ where ${\mathcal{D}}$ and ${\mathcal{D}}^\perp$ are complex and totally real distributions of $M$, respectively. Interchanging $X$ by $JX$, then we get the relation (4.4) of Theorem 4.2 in \cite{C3}.
\vskip.1in

\noindent 2. Also, if $\theta_1=0$ and $\theta_2=\theta$, a slant function, then the submanifold $M$ becomes pointwise semi-slant which has been studied in \cite{S3}. In this case, the first two terms in the left hand side of (\ref{4.12}) vanish identically. Thus, the relation (\ref{4.12}) is true for pointwise semi-slant warped product and it will be $$A_{FTZ}X-A_{FZ}JX=-\left(\sin^2\theta\right)X(\mu)\,Z,\,\,\,X\in {\mathcal{D}},\,\,\,Z\in {\mathcal{D}}_\theta,$$ where ${\mathcal{D}}$ and ${\mathcal{D}}_\theta$ are complex and proper pointwise slant distributions of $M$. Hence, Theorem 5.1 of \cite{S3} is a special case of Theorem  \ref{T:3}. In fact, in the relation (5.4) of Theorem 5.1 in \cite{S3}, the term $(1+\cos^2\theta)$ should be $(1-\cos^2\theta)$, i.e., there is a missing term.
\vskip.1in

\noindent 3. If we consider $\theta_1=\theta$, a constant slant angle and $\theta_2=\frac{\pi}{2}$, then it is a case of hemi-slant warped products which have been discussed in \cite{S2}. In this case, the second and third terms in the left hand side of (\ref{4.12}) vanish identically. Hence, (\ref{4.12}) is valid for hemi-slant warped products. Thus, the Theorem  \ref{T:3} is also a generalization of Theorem 5.1 in \cite{S2}. In this case the relation (\ref{4.12}) will be $$A_{FTX}Z-A_{JZ}TX=-(\cos^2\theta) X(\mu)\,Z,\,\,\,X\in {\mathcal{D}}_\theta,\,\,\,Z\in {\mathcal{D}}^\perp,$$where ${\mathcal{D}}_\theta$ and ${\mathcal{D}}^\perp$ are proper slant and totally real distributions. Hence Theorem 5.1 of \cite{S2} can be proved without using the mixed totally geodesic condition.
\vskip.1in

\noindent 4. In Theorem  \ref{T:3} if we assume $\theta_1=\frac{\pi}{2}$ and $\theta_2=\theta$  a pointwise slant function,  then this is the case of pointwise hemi-slant warped products  studied in \cite {U4}. In this case  (\ref{4.12}) reduces to the form $$A_{FTZ}X-A_{JX}TZ=(\cos^2\theta)X(\mu)\,Z,\,\,\,X\in {\mathcal{D}}^\perp,\,\,\,Z\in {\mathcal{D}}_\theta,$$where ${\mathcal{D}}^\perp$ and ${\mathcal{D}}_\theta$ are totally real and proper pointwise slant distributions of a pointwise hemi-slant submanifold $M$ in a Kaehler manifold $\tilde M$, which a condition of Theorem 4.2 in \cite{U4}. Therefore Theorem  \ref{T:3} is also a generalized version of Theorem 4.2 in \cite{U4}.

\begin{remark} The inequality for the squared norm of the second fundamental form of a warped product pointwise bi-slant submanifold can be evaluated by using only the mixed totally geodesic condition. And, if the warped product is mixed totally geodesic, then by Theorem \ref{T:1}, either it is a Riemannian product or a warped product pointwise hemi-slant submanifold of the form $M_\theta\times_fM_\perp$, where $M_\theta$ is a proper pointwise slant submanifold and $M_\perp$ is a totally real submanifold of a Kaehler manifold $\tilde M$. These kinds of warped products are special case of hemi-slant warped products which have been considered in \cite{S2} and the inequality is obtained by using the mixed totally geodesic condition.
\end{remark}

\end{document}